\documentclass{article}

\usepackage[a4paper, total={7in, 8in}]{geometry}
\usepackage[utf8]{inputenc}   
\usepackage[T1]{fontenc}      
\usepackage{multicol}
\usepackage{amsmath,amsthm} 
\usepackage{amssymb,mathrsfs} 
\usepackage{amssymb}
\usepackage{amsfonts}
\usepackage[normalem]{ulem}

\usepackage{graphicx}
\usepackage{enumerate}
\usepackage{graphicx} 
\usepackage{lmodern} 
\usepackage{mathtools}
\usepackage{dsfont}
\usepackage{physics}
\usepackage{subfig}
\usepackage{framed}
\usepackage{color}

\newcommand\N{\mathbb{N}}
\newcommand\E{\mathbb{E}}

\newcommand\R{\mathbb{R}}

\newcommand\proba{\mathbb{P}}
\newcommand\e{\mathrm{e}}
\renewcommand\liminf{\underline{\lim}}
\renewcommand\limsup{\overline{\lim}}

\newcommand\ind{\mathds{1}}

\newcommand\eps{\varepsilon}
\newcommand\alphap{{\alpha_p}}

\newcommand{\vertiii}[1]{{\left\vert\kern-0.25ex\left\vert\kern-0.25ex\left\vert #1 
    \right\vert\kern-0.25ex\right\vert\kern-0.25ex\right\vert}}

\renewcommand{\leq}{\leqslant}

\renewcommand{\geq}{\geqslant}

\newtheorem{theorem}{Theorem}
\newtheorem{lemma}{Lemma}

\iffalse    
   \excludecomment{corrige}
\else
  
\fi

\begin{document}

\title{Heavy tailed large deviations for time averages of
  diffusions: the Ornstein--Uhlenbeck case}
\author{Grégoire Ferré\\
\small Capital Fund Management, 23-25 rue de l'Université, 75007 Paris
}

\date{\today}

\maketitle

\abstract{
  We study large deviations for the time average of the Ornstein--Uhlenbeck process
  raised to an arbitrary power. We prove that beyond a critical value, large deviations
  are subexponential in time, with a non-convex rate function whose main coefficient
  is given by the solution to a Hamilton--Jacobi problem. Although a similar problem was
  addressed in a recent work, the originality of the paper
  is to provide a short, self-contained proof of this result through a couple of standard
  large deviations arguments.
}

\section{Introduction}

Describing the long time behaviour of empirical averages of diffusions is a recurrent theme
in probability theory~\cite{dembo2010large} but also in other fields
such as statistical physics~\cite{touchette2009large}.
While the long time convergence  can be treated via various ergodic
results~\cite{lelievre2016partial}, small
fluctuations around the mean are described by Central Limit Theorems related to
 Poisson equations~\cite{bhattacharya1982functional}, while large
fluctuations are described by so-called large deviations
results~\cite{dembo2010large,ellis2007entropy}. However, since the founding
papers~\cite{donsker1975asymptoticI},
most large deviations results are still derived when
considering the mean of bounded functions. This is unsatisfactory since most
real-life situations deal with unbounded quantities.

A series of recent papers try to address fluctuations
of unbounded observables. This starts with pioneering works concerned with particular
cases such as the Langevin dynamics,
see~\cite{wu2001large} and references therein. To our knowledge, the most general
result in this direction is~\cite{ferre2019large}. This work provides the largest
set of unbounded functions
satisfying standard large deviations asymptotics for a given diffusion. This can be alternatively
viewed as a spectral gap results for unbounded non-self adjoint operators in
Wasserstein-like topologies, generalizing~\cite{donsker1975variational,gartner1977large}
to unbounded state spaces and unbounded observables in a Wasserstein class.
However these results only cover the normal case when large deviations are
exponentially small in time (\textit{i.e.} a spectral gap exists~\cite{ferre2021more})
and do not go beyond.

Outside the realm of the spectral gap case, very few results exist. A first
breakthrough paper partially describes fluctuations for powers of the Ornstein--Uhlenbeck
process through low temperature approximation
arguments~\cite{nickelsen2018anomalous}. To the knowledge
of the author,~\cite{bazhba2022large} is the first full-proved
result in this direction. In this paper, a large deviations principle
(LDP) is proved for a generalized
 Ornstein--Uhlenbeck process raised to some power. The key idea is
to split the dynamics into excursions out of the origin, leading to a small temperature
dynamics similar to the one hinted at in~\cite{nickelsen2018anomalous}.

However, the proofs in~\cite{bazhba2022large} are quite involved and might be difficult to
generalize to other diffusions. The goal of the current paper is to cast ourselves
in a slightly simpler situation (non-restrictive from a tail point of view)
to propose a proof based on a couple of natural large deviations
arguments. In particular, when reducing the dynamics to a sum of independent excursions,
we show that the techniques used
in~\cite{gantert2014large,ferre2022subexponential} are naturally leveraged with
small temperature large deviations~\cite{freidlin1998random}.
With our technique, we retrieve that large deviations are visible at a subexponential scale
with a non-convex, subexponential rate function
whose pre-factor is given by a variational problem.
This result is very similar to the one obtained in~\cite{ferre2022subexponential}
for independent variables, in contrast with the usual case where the
rate function is given by a Fenchel transform~\cite{ferre2019large}.
Moreover, it suggests that the approximation performed in~\cite{nickelsen2018anomalous}
is actually exact for long times, while corrections in finite time could be estimated
via expansion techniques~\cite{nickelsen2022noise,ferre2023stochastic}.

Although it seems that the result is not new~\cite{bazhba2022large},
we believe the simplicity of the proposed
approach provides a better understanding of heavy-tailed large deviations and will allow to derive
more general results.

\vspace{0.5cm}


\section{Main result}
\label{sec:main}

We consider the one dimensional SDE on~$\R_+$:
\begin{equation}
  \label{eq:X}
  dX_t = - \gamma X_t \,dt + dW_t,\quad  \quad X_0=0,
\end{equation}
where $(W_t)_{t\geq0}$ is a real-valued standard Brownian motion and $\gamma >0$
is a parameter fixed throughout. We investigate
large deviations for the integral
\[
L_T^p = \frac1T\int_0^T f_p(X_t)\,dt,
\]
where $T>0$ and, for $p>0$, we use the scaling function
\begin{equation}
  \forall\,x\in\R,\quad f_p(x) =\mathrm{sign}(x) |x|^p.
\end{equation}
For $0<p\leq2$ we know that~$(L_T^p)_{T> 0}$ obeys  large deviations
asymptotics at exponential scale when $T\to+\infty$
with a smooth, good rate function that can be expressed as the Fenchel transform of a cumulant
generating function~\cite{touchette2009large,dembo2010large}. When $p>2$, the only result known by
the author is shown in~\cite{bazhba2022large}.
We propose in this paper a full LDP proof with simple self-contained arguments inspired
by~\cite{gantert2014large,ferre2022subexponential}.
With this alternative basis relying on standard large deviations techniques, we hope
to build more general results in the future.

Before stating our main result, we introduce $C(I)$ (resp.~$\mathscr{C}(I)$)
the set of continuous  (resp. absolutely continuous) functions on an interval~$I\subset\R$.
When~$I = [0,b]$ or $I=[0,b)$ for $b>0$, we write
$C_x(I)=\{\varphi\in C(I),\ \varphi_0 = x\}$ and similarly for~$\mathscr{C}_x$.
We also introduce the Freidlin--Wentzell functional associated with the diffusion~\eqref{eq:X}
namely, for any~$T>0$:
\begin{equation}
  \label{eq:J}
  \forall\, \varphi \in C([0,T]),\quad
  \mathcal{J}_T(\varphi)=\left\{ \begin{aligned}
    &\frac12
    \int_0^T | \dot\varphi_t + \gamma \varphi_t|^2\,dt, & \ \mbox{if}\ \varphi\in\mathscr{C}_0
    ([0,T]),
    \\ &+\infty, &\quad \mbox{otherwise}.
  \end{aligned}
  \right.
\end{equation}
In all what follows, we define the crucial exponent:
\[\alphap = \frac2p.\]
Note that~$\alphap \in(0,1)$ when $p>2$. Our main result is as follows.
\begin{theorem}
  \label{th:main}
  For $p>2$, $(L_T^p)_{T>0}$ satisfies a large deviations principle in~$\mathbb{R}$
  at speed~$T^{\alphap}$
  and with rate function
  \[
  \forall\,x\in\R,\quad
  I(x) = J_\infty^\star |x|^{\alphap},
  \]
  where
  \[
  J_\infty^\star =\underset{\varphi\in\mathscr{C}_0([0,+\infty))}{ \inf}
  \left\{ \mathcal{J}_\infty(\varphi),\ \mbox{with}\
    \int_0^\infty f_p(\varphi_t)\,dt = 1\right\}.
\]
\end{theorem}

\section{Proof of Theorem~\ref{th:main}}
\label{sec:proof}
We proceed very much in the spirit of~\cite{gantert2014large,ferre2022subexponential}
for the independent case.
For this, instead of isolating a single variable, we extract a finite time
interval of the trajectory. This turns the long time problem into a small temperature
one that we can study through Freidlin--Wentzell asymptotics. To the knowledge
of the author, this technique was introduced
in~\cite{nickelsen2018anomalous} as an approximation and used further in~\cite{bazhba2022large}.

In all what follows we use the applications:
\begin{equation}
\label{eq:FH}
  F_H^p : \varphi \in C^0([0,H]) \mapsto \int_0^H f_p(\varphi_t)\,dt,
\quad \mbox{and}\quad \overline F_H^p : \varphi \in C^0([0,H]) \mapsto \int_0^H |\varphi_t|^p\,dt,
\end{equation}
where $H>0$ is arbitrary. Lemma~\ref{lem:continuous} in Appendix~\ref{sec:instanton}
shows that~$F_H^p$ and~$\overline F_H^p$ are continuous applications
on~$C^0([0,H])$, an important property to manipulate Freidlin--Wentzell asymptotics.

\subsection{Lower bound}
\label{sec:lower}

As we said, our strategy is to isolate a finite fraction of time over which
the fluctuation should realize, and show that this leads to
a low temperature behaviour leading to the asymptotics, similar to what is done
in~\cite{bazhba2022large}.
We recall that it is sufficient to prove the LDP lower bound on an arbitrary open ball.
Let us thus consider $x,\delta,\eps, H>0$ and write
\[
\begin{aligned}
\proba\big(
L_T^p\in (x-\delta,x+\delta)
\big)& =  \proba\left(
x-\delta < \frac1T \int_0^T f_p(X_s)\,ds < x +\delta\right)
\\ & \geq\proba\left(x-\delta-\eps < \frac1T \int_0^H f_p(X_s)\,ds < x +\delta+\eps,
\ -\eps \leq \frac1T \int_H^T f_p(X_s)\,ds \leq \eps\right)
\\ & \geq \proba\left(x-\delta-\eps < \frac1T \int_0^H f_p(X_s)\,ds < x +\delta+\eps\right)
\proba\left(-\eps \leq \frac1T \int_H^T f_p(X_s)\,ds \leq \eps\right).
\end{aligned}
\]
The last probability goes to one when $T\to +\infty$ and $H>0$ is fixed
by a standard ergodicity argument. We thus focus on:
\[
\proba\left(x-\delta-\eps < \frac1T \int_0^H f_p(X_s)\,ds < x +\delta+\eps\right)
= \proba\left(x_- < \int_0^H f_p(X_s^T)\,ds < x_+\right),
\]
where we used the shorthand notation $x_\pm = x \pm(\delta +\eps)$ and
introduced the process~$(X_t^T)_{t\in[0,H]}$ defined by:
\begin{equation}
  \label{eq:Xscaled}
\forall\, t \in[0,H],\quad  X_t^T = \frac{X_t}{T^{\frac1p}}.
\end{equation}
We leave the proof of the following simple lemma to the reader.
\begin{lemma}
  \label{lem:XT}
  The process~$(X_t^T)_{t\in[0,H]}$ satisfies the following SDE on~$[0,H]$:
  \begin{equation}
    \label{eq:XT}
    dX_t^T = - \gamma X_t^T\,dt + \eps_T \sigma\,dW_t,\quad \eps_T = \frac{1}{T^{\frac1p}},
    \quad X_0^T = 0.
  \end{equation}
\end{lemma}
The meaning of this lemma is that, by imposing the fluctuation to take place within
a finite time window~$[0,H]$, we naturally exhibit a small temperature problem.

Since the mapping~$F_H^p$ defined in~\eqref{eq:FH} is continuous, the set
$\{\varphi\in C^0([0,H]),\ x_-< F_H^p(\varphi)<x_+\}$ is open.
By using the Freidlin--Wentzell asymptotics 
Theorem~\ref{th:FW} recalled in Appendix~\ref{sec:instanton}, we thus have
\[
\underset{T\to\infty}{\liminf}
\ \eps_T^2 \log \proba\left(x_- < \int_0^H f_p(X_s^T)\,ds < x_+\right)
\geq - \inf_{\varphi\in \mathscr{C}_0([0,H])}\left\{
\frac12 \int_0^H |\dot \varphi_t + \gamma \varphi_t|^2\,dt,\quad \int_0^H f_p(\varphi_t)\,dt \in (x_-,x_+)
\right\}.
\]
Following the notations in~\eqref{eq:J} and~\eqref{eq:FH} and
dividing the arbitrary variable~$\varphi$
by~$x^{\frac1p}$ in the infimum, we obtain
\begin{equation}
  \label{eq:eqintJ}
\underset{T\to\infty}{\liminf}
\ \eps_T^2 \log \proba\left(x_- < \int_0^H f_p(X_s^T)\,ds < x_+\right)
\geq - x^{\frac2p} \inf_{\varphi\in \mathscr{C}_0([0,H])}\left\{
\mathcal{J}_H(\varphi),\quad F_H^p(\varphi) \in \left(1-\frac{\delta+\eps}{x},1+\frac{\delta+\eps}{x}\right)
\right\}.
\end{equation}
Since~$\mathcal{J}_H$ is a good rate function and~$F_H^p$ is continuous, we can
use~\cite[Chapter~II,~Lemma~2.1.2]{deuschel2001large} to obtain that
\[
\underset{\eps,\delta\to0}{\lim}\inf_{\varphi\in \mathscr{C}_0([0,H])}\left\{
\mathcal{J}_H(\varphi),\quad \int_0^H f_p(\varphi_t)\,dt \in \left(1-\frac{\delta+\eps}{x},1+\frac{\delta+\eps}{x}\right)
\right\}= \inf_{\varphi\in \mathscr{C}_0([0,H])}\left\{
\mathcal{J}_H(\varphi),\quad \int_0^H f_p(\varphi_t)\,dt =1
\right\}.
\]
We finally show in Appendix~\ref{sec:cutting} that
\[
\underset{H\to\infty}{\lim}
\inf_{\varphi\in \mathscr{C}_0([0,H])}\left\{
\mathcal{J}_H(\varphi),\quad \int_0^H f_p(\varphi_t)\,dt=1
\right\} = J_\infty^\star.
\]
As a result,~\eqref{eq:eqintJ} becomes:
\[
\underset{T\to\infty}{\liminf}
\ \eps_T^2 \log \proba\left(x_- < \int_0^H f_p(X_s^T)\,ds < x_+\right)
\geq - x^{\frac2p}J_\infty^\star,
\]
which concludes the proof of the lower bound.

\subsection{Upper bound}
\label{sec:upper}

For the upper bound we again follow the  proof of the independent
variables case~\cite{ferre2022subexponential}, but using excursions
of the process $(X_t)_{t\in[0,T]}$ like in~\cite{bazhba2022large}.
For a fixed~$\eps_0$, we define $\tau_0=0$ and, for $i>1$, we set
\[
\tau_i^{\eps_0} = \inf\{t\geq \tau_{i-1},\ X_t = \eps_0\},\quad
\tau_i = \inf\{t\geq \tau_{i}^{\eps_0},\ X_t = 0\}.
\]
In other words, we first ensure that the process has departed away from~$0$, and then consider its
excursion back to the origin.
These are clearly stopping times adapted to $X_t$, and we can use the decomposition
\[
\frac1T \int_0^T f_p(X_t)\, dt =  \sum_{i=1}^{N_T} \int_{\tau_{i-1}}^{\tau_i} f_p(X_s^T)\,ds + \int_{\tau_{N_T}}^Tf_p(X_s^T)\,ds,
\]
where the random variable $N_T$ is the number of cycles up to time~$T$:
\begin{equation}
  \label{eq:NT}
  N_T = \max\left\{ k\geq 0,\ \sum_{i=1}^k \tau_i \leq T\right\}.
\end{equation}
Since at each cycle the process reaches the origin back, the random variables
\begin{equation}
  \label{eq:CiT}
C_i^T = \int_{\tau_{i-1}}^{\tau_i} f_p(X_s^T)\,ds
\end{equation}
are independent and identically distributed, and we have
\[
\frac1T \int_0^T f_p(X_t)\, dt  =
\sum_{i=1}^{N_T} C_i^T + \widetilde C^T,
\]
where
\[
\widetilde C^T = \int_{\tau_{N_T}}^Tf_p(X_s^T)\,ds
\]
is a remainder that is clearly neglictible with respect to the sum (we leave the proof
of this assertion to the reader and neglect this term bellow).
We thus split the trajectory into its~$N_T$ excursions and we proceed like in
the i.i.d. situation by first considering the case
where one excursion realizes the fluctuation, and the case where the whole trajectory
with no particular excursion does.

In order to do so, we also need to control the number of excursions.
Since $N_T\sim T\ \E[\tau]$ we introduce
$M_T=\lceil T\ \E[\tau]\rceil$ and the familly of open intervals
\[
\forall\, T>0,\ \forall\, \bar \eps >0, \quad \mathcal{M}_{T,\bar \eps}=\left( T \left(\frac{1}{\E[\tau]}- \bar \eps\right),
  T \left(\frac{1}{\E[\tau]}+ \bar \eps\right)\right),
\]
which allows to write
\begin{equation}
  \label{eq:decomproba}
\begin{aligned}
\proba\left(
\frac1T \int_0^T f_p(X_t)\, dt \geq x
\right) & \leq \proba\left(\sum_{i=1}^{N_T} C_i^T\geq x
\right)\leq \proba\left(\sum_{i=1}^{M_T} C_i^T\geq x,\ N_T \in \mathcal{M}_{T,\bar \eps}
\right)+ \proba\left(N_T \notin \mathcal{M}_{T,\bar \eps}\right)
\\ &\leq \proba\left(
\exists \,j \in\{1,\hdots,M_T\},\ C_j^T \geq x
\right)
\\ & + \proba\left(
\forall \,j \in\{1,\hdots,M_T\},\ C_j^T < x \ \mbox{and}
\ \sum_{i=1}^{M_T} C_i^T \geq x
\right)+ \proba\left(N_T \notin \mathcal{M}_{T,\bar \eps}\right).
\end{aligned}
\end{equation}
In what follows, the last probability on the right hand side above can be neglected  since
we show in Appendix~\ref{sec:MT} that
\[
\underset{T\to\infty}{\limsup}\ \frac{1}{T^{\alphap}}
\log \proba(N_T\notin \mathcal M_{T,\bar \eps} )= -\infty.
\]

The behaviour of~$C_1^T$ as $T\to +\infty$ is crucial for treating
the two remaining terms. In Appendix~\ref{sec:CT} we prove the following crucial
subexponential large deviations upper bound for one variable,
which will be used at various places below.

\begin{lemma}
  \label{lem:CT}
  For any $x>0$ it holds
  \begin{equation}
    \underset{T\to\infty}{\limsup}\ \eps_T^2 \log\proba\left(C_1^T \geq x\right)
    \leq - x^{\alphap} J_\infty^\star.
  \end{equation}
\end{lemma}

We next study the two cases that arise, in a fashion
similar to~\cite{gantert2014large,ferre2022subexponential}.

\subsection*{The heavy tail component}
By the union's bound we have
\[
\proba\left(
\exists \,j \in\{1,\hdots,M_T\},\ C_j^T \geq x
\right) \leq M_T \proba\left( C_1^T \geq x
\right).
\]
Since $M_T=\lceil T\ \E[\tau]\rceil$ with~$\Gamma>0$ and $\eps_T = T^{-\frac1p}$, we have
$\eps_T^2 \log \left(M_T\right)\xrightarrow[]{T\to\infty} 0$. Therefore, 
Lemma~\ref{lem:CT} above implies that
\[
\begin{aligned}
\underset{T\to\infty}{\limsup}
\ \eps_T^2 \log\proba\left(
\exists \,j \in\{1,\hdots,M\},\ C_j^T \geq x
\right) &
\leq \underset{T\to\infty}{\limsup}\left[ \eps_T^2 \log\proba\left(
  C_1^T \geq x\right) 
+\eps_T^2 \log \left(M_T\right)\right)
\\ & \leq - x^{\alphap} J_\infty^\star,
\end{aligned}
\]
which is the expected result for the first probability on the last line of~\eqref{eq:decomproba}.

\subsection*{Light tail part}
Let's turn to the second term, which encompasses the lighter tail part of the
asymptotics. As opposed to the direct analytical method used in~\cite{bazhba2022large},
we rely on the couple of arguments used in~\cite{gantert2014large,ferre2022subexponential}.
First, for any~$\eta_T>0$ we use Tchebychev's inequality:
\[
\begin{aligned}
\underbrace{ \proba\left(
\forall \,j \in\{1,\hdots,M_T\},\ C_j^T < x \ \mbox{and}\
\ \sum_{i=1}^{M_T} C_i^T \geq x
\right)}_{(P)} & \leq  \e^{-\eta_T x} \E\left[
  \ind_{\left\{\forall\,j\in\{1,\hdots,M_T\},\, C_j^T<x\right\} }
\e^{\eta_T\sum_{i=1}^{M_T} C_i^T}
\right]\\ &
\leq \e^{-\eta_T x} \E\left[
  \ind_{\left\{ C_1^T<x\right\} }
\e^{\eta_T C_1^T}\right]^{M_T}.
\end{aligned}
\]
where we used independence of the~$C_i$'s for moving to the second line.
It is natural to set~$\eta_T = \theta \eps_T^{-2}$
for some $\theta>0$. Moreover, recalling
that $X_t^T = X_t/ T^{\frac1p}$, $C_1^T = C_1/T$,
and $\eps_T=T^{-\frac1p}$ (since $M_T = \lceil T/\E[\tau]\rceil$
we introduce $\xi_T = M_T/T\xrightarrow[]{T\to\infty} 1$) we obtain
\[
\eps_T^2\log (P) \leq - \theta x +\xi_T T^{1 - \alphap} \log\left( \E\left[
  \ind_{\left\{  C_1<Tx\right\} }
\e^{\theta T^{\alphap-1}C_1}
  \right] \right).
\]
In order to conclude we will prove the following lemma.
\begin{lemma}
  \label{lem:negative}
  For any  $\theta <x^{\alphap-1} J_\infty^\star$ it holds
  \[
\underset{T\to\infty}{\limsup}\ T^{1-\alphap} \log \E\left[
  \ind_{\left\{  C_1<Tx\right\} }
\e^{\theta T^{\alphap-1}C_1}
  \right]\leq 0. 
  \]
\end{lemma}

We will closely follow the techniques
of~\cite{gantert2014large}, adapting a couple of arguments along~\cite{ferre2022subexponential}
with Freidlin--Wentzell asymptotics. For this, let us recall that 
$\log y \leq y - 1$ for $y>0$ and, for any integer~$k$, we have
$\e^y - 1 \leq y +y^2/2 +\hdots+\e^y y^{k+1}/(k+1)!$.
Therefore, for an arbitrary $k\in \N^*$,
\begin{equation}
  \label{eq:expand}
T^{1-\alphap}\log 
\E\left[
  \ind_{\left\{  C_1<Tx\right\} }
\e^{\theta T^{\alphap-1}C_1}
\right] \leq T^{1-\alphap} \sum_{j=1}^k\E\left[\ind_{\left\{  C_1<Tx\right\} }
\frac{(\theta T^{\alphap-1}C_1)^j}{j!}  \right] + \frac{R_T}{(k+1)!},
\end{equation}
where
\[
R_T = \theta^{k+1} T^{1-\alphap + (k+1)(\alphap - 1)}\E\left[
  \ind_{\left\{  C_1<Tx\right\} }
C_1^{k+1}\e^{\theta T^{\alphap-1}C_1}
\right].
\]
Let us first consider the sum on the right hand side of~\eqref{eq:expand}.
For $j=1$, we have $\E[C_1]=0$ so the first term is zero.
Then, for $j>1$ we can show that $\E[|C_1|^j]<+\infty$ (as powers and
integrals of Gaussian distributions). Therefore, for any $j> 1$, it holds
\[
T^{1-\alphap} \left|\E\left[\ind_{\left\{  C_1<Tx\right\} }
\frac{(\theta T^{\alphap-1}C_1)^j}{j!}  \right]\right|\leq
\frac{T^{(\alphap - 1)(j-1)}\theta^j}{j!} \E\left[ |C_1|^j  \right]
\xrightarrow[T\to+\infty]{}0.
\]
As a result, the sum in~\eqref{eq:expand} is asymptotically bounded by~$0$
as $T\to+\infty$, so it is negligible at subexponential scale.

Let us turn to the remainder~$R_T$ by using Holder's inequality,
\begin{equation}
  \label{eq:RTdecomp}
R_T \leq \left(\theta^{k+1} T^{ 1 +(k+1)(\alphap - 1)}\E\Big[
  \ind_{\left\{  C_1<Tx\right\} }
|C_1|^{(k+1)r}
\Big]^{\frac1r}\right)\left( T^{-\alphap} \E\left[
  \ind_{\left\{  C_1<Tx\right\} }
\e^{q\theta T^{\alphap-1}C_1}
\right]^{\frac1q}\right),
\end{equation}
where $r,q>1$ are arbitrary real numbers satisfying $1/r + 1/q = 1$.
If we choose $k$ large enough such that
\[
k > \frac{\alphap}{1-\alphap},
\]
then it holds
\[
1 + (k+1)(\alphap - 1)< 0.
\]
Since $\E\left[|C_1|^{(k+1)r}\right]<+\infty$ for any choice of~$k,r>0$,
the first term in the right hand side of~\eqref{eq:RTdecomp} thus goes to zero when
$T\to+\infty$.

Let us now show that the second term on the right hand side of~\eqref{eq:RTdecomp} satisfies,
for some $q>1$:
\begin{equation}
  \label{eq:RTbound}
\underset{T\to\infty}{\limsup}\ T^{-\alphap} \E\left[
  \ind_{\left\{  C_1<Tx\right\} }
\e^{q\theta T^{\alphap-1}C_1}
\right]^{\frac1q} < +\infty.
\end{equation}
For this, we use the integration by part formula~\cite[Lemma~4.5]{gantert2014large} to
obtain that
\[
T^{-\alphap}\E\left[
  \ind_{\left\{  C_1<Tx\right\} }
\e^{q\theta T^{\alphap-1}C_1}
\right]\leq q \theta T^{ - 1}\int_0^{Tx} \e^{q\theta T^{\alphap-1}z}\,\proba(C_1\geq z)\,dz.
\]
We use the change of variable $Txy=z$ to obtain (recall that $C_1^T = C_1/T$):
\[
q \theta T^{ - 1}\int_0^{Tx} \e^{q\theta T^{\alphap-1}z}\,\proba(C_1\geq z)\,dz
= q \theta x\int_0^{1} \e^{q\theta T^{\alphap}xy}\,\proba(C_1^T\geq xy)\,dy.
\]
In order to conclude the proof, we show that the term within the integral on the right
hand side is bounded for any $y\in[0,1]$. Indeed, let $y\in[0,1]$ and write
\[
\frac{1}{T^\alphap}\log\left[\e^{q\theta T^{\alphap}xy}\,\proba(C_1^T\geq xy)\right]
= q\theta xy + \frac{1}{T^\alphap}\log\proba(C_1^T\geq xy).
\]
Since any sequence is bounded by its superior limit, we can leverage
Lemma~\ref{lem:CT} to obtain
\begin{equation}
  \label{eq:limsupb}
\frac{1}{T^\alphap}\log\left[\e^{q\theta T^{\alphap}xy}\,\proba(C_1^T\geq xy)\right]
\leq 
  q\theta xy +\underset{T\to+\infty}{\limsup} \ \frac{1}{T^\alphap}\log\proba(C_1^T\geq xy)
\leq q\theta xy - (yx)^\alphap J_\infty^\star.
\end{equation}
Since~$\theta< x^{\alphap-1} J_\infty^\star$ and $q>1$ are arbitrary, we can let~$\eps_\theta>0$
and set
\[
\theta = (1-\eps_\theta)^2x^{\alphap-1} J_\infty^\star, \quad
q = \frac{1}{1-\eps_\theta}.
\]
This leads to
\[
q\theta xy - (yx)^\alphap J_\infty^\star = x^\alphap J_\infty^\star \left(
(1-\eps_\theta)y - y^\alpha
\right).
\]
Since $\alphap<1$, we have
$(1-\eps_\theta)y - y^\alphap\leq 0$ for any $y\in[0,1]$, so~\eqref{eq:limsupb}
ensures that 
\[
\forall\,y\in[0,1], \quad \e^{q\theta T^{\alphap}xy}\,\proba(C_1^T\geq xy)\leq 1.
\]
This provides the bound~\eqref{eq:RTbound}, so Lemma~\ref{lem:negative}
holds and the proof of the upper bound is complete.


\section*{Appendix}
\appendix

\section{Freidlin--Wentzell asymptotics and the rate function}
\label{sec:instanton}

In this section, we recall the Freidlin--Wentzell low temperature asymptotics
and prove a couple of results on the associated rate function. 
For this we consider the small temperature equivalent of~\eqref{eq:X} that arises
in Lemma~\ref{lem:XT}, for a fixed $H>0$:
\begin{equation}
  \label{eq:Xeps}
    dX_t^\eps = - \gamma X_t^\eps \,dt +\eps dW_t,\quad X_0^\eps=0,\quad t\in[0,H],
\end{equation}
for any $\eps>0$. When the parameter~$\eps$ becomes small, the following result
holds~\cite[Chapter~IV, Th.~1.1]{freidlin1998random}.
\begin{theorem}
\label{th:FW}
  Let $H>0$, then for any measurable set $A\subset C^0([0,H])$:
\begin{equation}
  \label{eq:FWasymp}
  \inf_{\varphi \in  \mathring{A}}\ \mathcal{J}_H(\varphi) \leq
  \underset{\eps\to 0}{\liminf} \ \eps^2 \log \proba\big( (X_t^\eps)_{t\in[0,H]}\in A\big)\leq
\underset{\eps\to 0}{\limsup} \ \eps^2 \log \proba\big( (X_t^\eps)_{t\in[0,H]}\in A\big)
\leq \inf_{\varphi \in  \overline{A}}\ \mathcal{J}_H(\varphi),
\end{equation}
where~$\mathring{A}$ and~$\overline{A}$ denote respectively the interior and closure
of~$A$ for the $C^0$-topology.
\end{theorem}

We now prove a couple of useful results on the rate function and the associated
optimization problem under constraint. We start with a continuity result on the constraint
function.
\begin{lemma}
  \label{lem:continuous}
  For any~$H,p>0$, the applications~$F_H^p$ and~$\overline F_H^p$ defined
  in~\eqref{eq:FH} are continuous on~$C^0([0,H])$.
\end{lemma}

\begin{proof}
  Consider a sequence~$(\varphi^n)_{n\in\N}$ in~$C^0([0,H])$ such
  that $\varphi^n\xrightarrow[n\to +\infty]{C^0}\varphi$
  for some~$\varphi\in C^0([0,H])$. We want to prove
  that~$F_H^p(\varphi^n)\xrightarrow[n\to +\infty]{} F_H^p(\varphi)$.
  Since $\varphi^n\xrightarrow[n\to +\infty]{C^0}\varphi$, for any~$\eps>0$
  there exists~$n_\eps$ such that, for all $n\geq n_\eps$,
  \begin{equation}
    \label{eq:continuity}
\sup_{t\in[0,H]}|\varphi_t^n - \varphi_t|\leq\eps.
  \end{equation}
  In particular, for $n\geq n_\eps$, it holds
  \[
  \forall\, t \in[0,H], \quad
  \underline \varphi - \eps\leq
\varphi_t^n
\leq\overline \varphi + \eps,
\]
where $\underline \varphi = \inf_{[0,H]}\varphi$ and $\overline \varphi
=  \sup_{[0,H]}\varphi$.
Now, the function~$f_p$ is continuous on~$\R$, so it is uniformly continuous
on $[\underline \varphi - \eps, \overline \varphi + \eps]$. For any~$\eps'>0$,
there is~$\delta_{\eps'} >0$ such that
\begin{equation}
  \label{eq:unifcont}
\forall \, (x,y)\in [\underline \varphi - \eps, \overline \varphi + \eps]^2
\quad \mbox{s.t.}\quad |x - y|<\delta_{\eps'},\quad
|f_p(x) - f_p(y)|\leq \eps'.
\end{equation}

Finally, we fix~$\eps'$, so there is~$\delta_{\eps'}>0$ such that~\eqref{eq:unifcont}
holds. By~\eqref{eq:continuity}, there is $n_{\delta_{\eps'}}$ such
that for $n\geq n_{\delta_{\eps'}}$ it holds that
$(\varphi_t^n, \varphi_t)\in [\underline \varphi - \eps, \overline \varphi + \eps]^2$
and by~\eqref{eq:continuity}, we obtain that
\[
\forall\, t \in[0,H],\quad
|\varphi_t^n - \varphi_t|\leq\delta_{\eps'}.
\]
In view of~\eqref{eq:unifcont} we also have
\[
\forall\, t \in[0,H],\quad
|f_p(\varphi_t^n) - f_p(\varphi_t)|\leq \eps'.
\]
Combining these elements, for $n\geq n_{\delta_{\eps'}}$ it holds
\[
|F_H^p (\varphi^n) - F_H^p(\varphi)|
\leq \int_0^H | f_p(\varphi_t^n) - f_p(\varphi_t)|\,dt
\leq H \eps'.
\]
This shows convergence and thus continuity of~$F_H^p$. A similar reasoning
holds for~$\overline F_H^p$.
    
\end{proof}

We now prove that the variational problem associated with the rate function can be
equivalently considered with~$F_H^p$ or~$\overline F_H^p$.
\begin{lemma}
  \label{lem:FW}
  Consider~$x,H>0$ and let~$\mathcal{J}_H$ be defined
  in~\eqref{eq:J}. Let us write
  \[
  J_H^\star(x)=\inf_{\varphi\in\mathscr{C}_0([0,H])}\left\{  \mathcal{J}_H(\varphi),\quad
  \int_0^H f_p(\varphi_t)\,dt \geq x \right\}.
  \]
and
 \[
\overline J_H^\star(x)=\inf_{\varphi\in\mathscr{C}_0([0,H])}\left\{  \mathcal{J}_H(\varphi),
 \quad \int_0^H |\varphi_t|^p\,dt \geq x \right\}.
 \]
 Then, for any $x>0$ it holds $J_H^\star(x) = \overline J_H^\star(x)$.
\end{lemma}

\begin{proof}
  In order to reach the desired result we show that
  a minimizer~$\varphi^\star$ of~$J_H^\star(x)$
  satisfies $\varphi^\star\geq0$ (note that such a minimizer exists by
  standard variational analysis arguments, or similarly consider an almost minimizer).
  We proceed by contradiction by assuming that~$\varphi^\star$ is such
  a minimizer for which there exists~$t_0\in[0,H]$ such that
  $\varphi_{t_0}<0$. We assume for simplicity that $t_0\neq H$ (a case
  easily deduced from the proof below) and note that $t_0\neq 0$ since $\varphi_0 = 0$.
  Since~$\varphi^\star$ is continuous, this means that
  $V = \{t\in[0,H], \, \varphi_t^\star <0\}$ is an open set with non-zero Lebesgue measure.

  Next, we can define 
  $\overline \varphi^\star = \max(\varphi^\star, 0)$ and observe
  that this function is absolutely continuous (because its absolute variations
  are smaller than the ones of~$\varphi^\star$).
  Moreover, it satisfies the constraint of the optimization problem since
  \[
  \int_0^H f_p(\overline \varphi^\star_t)\,dt \geq   \int_0^H f_p( \varphi^\star_t)\,dt
  \geq x.
  \]

  Let us then show that $\mathcal{J}_H(\overline \varphi^\star) < \mathcal{J}_H(\varphi^\star)$
  to obtain a contradiction. For this we note
  that $d\overline\varphi^\star_t/dt=\overline \varphi^\star_t=0$ Lebesgue
  almost-everywhere on~$V$ (and $\varphi^\star = \overline \varphi^\star$
  on $(0,H)\setminus V$) to write
  \[
  \begin{aligned}
  \mathcal{J}_H(\overline \varphi^\star) & = \frac12 \int_0^H
  \left|\frac{d}{dt}\overline \varphi^\star_t + \gamma\overline\varphi^\star_t\right|
\\ & = \frac12 \int_{[0,H]\setminus V}
  \left|\frac{d}{dt}\overline \varphi^\star_t + \gamma\overline\varphi^\star_t\right|
  \\ & = \frac12 \int_{[0,H]\setminus V}
  \left|\frac{d}{dt} \varphi^\star_t + \gamma\varphi^\star_t\right|
  \\ &  <\frac12 \int_{[0,H]}
  \left|\frac{d}{dt} \varphi^\star_t + \gamma\varphi^\star_t\right|.
\end{aligned}
  \]
  The last line comes from the fact that, if $t_0\in(t_-,t_+)\subset V$ is such that
  $\varphi^\star_{t_-}=\varphi^\star_{t_+}=0$ and $\varphi_{t_0}^\star<0$, it is
  impossible for the following condition to hold true:
  \[
 \frac{d}{dt} \varphi^\star_t = - \gamma \varphi^\star_t\quad \mbox{over}\quad (t_-, t_+),
  \]
which means that
  \[
\frac12 \int_{t_-}^{t_+}
  \left|\frac{d}{dt} \varphi^\star_t + \gamma\varphi^\star_t\right|^2>0.
  \]
  This entails that
  \[
  \mathcal{J}_H(\overline \varphi^\star) < \mathcal{J}_H(\varphi^\star),
  \]
  contradicting that~$\varphi^\star$ is a minimizer of~$\mathcal{J}_H$. We then deduce 
  that~$\varphi^\star\geq 0$, from which the desired result follows.
\end{proof}

\section{Proof of Lemma~\ref{lem:CT}}
\label{sec:CT}

We decompose over the two stopping times~$\tau_1^{\eps_0}$ and~$\tau_1$:
\[
\proba\left( C_1^T \geq x\right)
= \proba\left(\int_0^{\tau_1^{\eps_0}}f_p(X_s)\,ds
  + \int_{\tau_1^{\eps_0}}^{\tau_1}f_p(X_s)\,ds \geq Tx\right).
\]
Denoting by
\[
A^{\eps_0} = \int_0^{\tau_1^{\eps_0}}f_p(X_s)\,ds,
\]
and introducing some $\delta>0$, we have
\[
\begin{aligned}
\proba\left( C_1^T \geq x\right)=&
\proba\left(A^{\eps^0}
+ \int_{\tau_1^{\eps_0}}^{\tau_1}f_p(X_s)\,ds \geq Tx, \, A^{\eps^0}> \delta T\right)
+ \proba\left(A^{\eps^0}
+ \int_{\tau_1^{\eps_0}}^{\tau_1}f_p(X_s)\,ds \geq Tx, \, A^{\eps^0}\leq \delta T\right)
\\ &\leq \proba\left(\int_0^{\tau_1^{\eps_0}}f_p(X_s)\,ds > \delta T
\right)
+ \proba\left(
\int_{\tau_1^{\eps_0}}^{\tau_1}f_p(X_s)\,ds \geq (x-\delta)T\right)
\\ & = 
 \underbrace{\proba\left(\int_0^{\tau_1^{\eps_0}}f_p(X_s)\,ds > \delta T 
  \right)}_{(A)} +
\underbrace{\proba_{\eps_0}\left(
\int_{0}^{\tilde \tau}f_p(X_s^T)\,ds \geq (x-\delta)\right)}_{(B)},
\end{aligned}
\]
where
\[
\tilde \tau = \inf\,\{t\geq 0 \ \, \mbox{s.t.}\ \, X_t = 0,\ \mbox{with}\ X_0=\eps_0\}.
\]

Let us start with the term~$(A)$. For $s\in [0,\tau^{\eps_0}]$ it holds
$X_s \leq \eps_0$. As a result, we can use Tchebychev's inequality according to
\[
\proba\left(\int_0^{\tau_1^{\eps_0}}f_p(X_s)\,ds > \delta T \right)
\leq \proba\big( \eps_0^p \tau^{\eps_0}> \delta T \big)
\leq \E\left[\e^{\tau^{\eps_0}}\right] \e^{-\frac{\delta}{\eps_0^p}T}.
\]
Hence
\[
\underset{T\to +\infty}{\limsup}\, \eps_T^2 \log\proba
\left(\int_0^{\tau_1^{\eps_0}}f_p(X_s)\,ds > \delta T \right) = - \infty,
\]
and this term can be neglected.

We can now turn to the second term $(B)$ and introduce a time horizon~$H>0$. We note
that when~$X$ is started at~$\eps_0>0$, the process is positive before hitting~$\tilde\tau$.
Therefore
\[
\begin{aligned}
(B) = \proba_{\eps_0}\left(\int_{0}^{\tilde \tau}|X_s^T|^p\,ds \geq (x-\delta)\right)
 & \leq \proba_{\eps_0}\left(\int_{0}^{H}|X_s^T|^p\,ds \geq (x-\delta),
  \ \tilde \tau\leq H\right) + \proba(\tilde \tau > H).
  \\ & \leq \proba_{\eps_0}\left(\int_{0}^{H}|X_s^T|^p\,ds \geq (x-\delta)\right)
  + \proba(\tilde \tau > H).
\end{aligned}
\]
Since~$H$ is fixed, we can neglect $\proba(\tilde \tau > H)$ in the large~$T$
asymptotics. Moreover, for $\delta<x$,
Theorem~\ref{th:FW} and Lemma~\ref{lem:continuous} allow to obtain:
\[
\underset{T\to +\infty}{\limsup}\, \eps_T^2 \log\proba_{\eps_0}\left(
\int_{0}^{\tilde \tau}f_p(X_s^T)\,ds \geq (x-\delta)\right)
\leq - (x-\delta)^{\frac2p}\inf_{\varphi\in\mathscr{C}([0,H])}\left\{  \mathcal{J}_H(\varphi),
\quad \int_0^H |\varphi_t|^p\,dt \geq 1,\ \varphi_0=\eps_0 \right\}.
\]
By using Appendix~\ref{sec:cutting} below, this is equivalent to
\begin{equation}
  \label{eq:ineqeps0}
\underset{T\to +\infty}{\limsup}\, \eps_T^2 \log\proba_{\eps_0}\left(
\int_{0}^{\tilde \tau}f_p(X_s^T)\,ds \geq (x-\delta)\right)
\leq - (x-\delta)^{\frac2p}\inf_{\varphi\in\mathscr{C}([0,+\infty))}\left\{
  \mathcal{J}_\infty(\varphi),
\quad \int_0^\infty |\varphi_t|^p\,dt \geq 1,\ \varphi_0=\eps_0 \right\}.
\end{equation}
We thus want to show that
\[
\underbrace{\inf_{\varphi\in\mathscr{C}([0,+\infty])}\left\{  \mathcal{J}_{\infty}(\varphi),
  \quad \int_0^{\infty} |\varphi_t|^p\,dt \geq 1,\ \varphi_0=0 \right\}}_{(C)}
\leq
\underset{\eps_0\to0}{\lim}
\underbrace{\inf_{\varphi\in\mathscr{C}([0,+\infty])}\left\{  \mathcal{J}_\infty(\varphi),
\quad \int_0^\infty |\varphi_t|^p\,dt \geq 1,\ \varphi_0=\eps_0 \right\}}_{(C_{\eps_0})}.
\]
For this, let us consider~$\eps>0$ and an $\eps$-minimizer
$\varphi^{\eps_0}\in C_{\eps_0}([0,+\infty))$ of~$(C_{\eps_0})$, namely
\[
\mathcal{J}_\infty(\varphi^{\eps_0}) < (C_{\eps_0}) + \eps, \quad
\int_0^\infty |\varphi_t^{\eps_0}|^p\,dt \geq 1.
\]
We next define
\[\forall\,t\in[0,\infty],\quad
\hat \varphi_t^{\eps_0} = \eps_0 t \mathds{1}_{\{ t \leq 1 \}}
+ \varphi^{\eps_0}_{t-1}\mathds{1}_{\{ t \geq 1 \}}.
\]
Since $\hat \varphi \in C_0([0,H])$ it holds
\[\begin{aligned}
(C)
&\leq \mathcal{J}_\infty(\hat \varphi^{\eps_0})
\\ &= \frac12 \int_0^1 |\eps_0 + \gamma \eps_0 t|^2\,dt + \frac12
\int_1^\infty|\dot \varphi^{\eps_0}_{t-1} +\gamma\varphi^{\eps_0}_{t-1}|^2
\,dt
\\  & = \frac{\eps_0^2}{2}\left(1 + \frac43\gamma\right)
+\int_0^\infty|\dot \varphi^{\eps_0}_{t'} +\gamma\varphi^{\eps_0}_{t'}|^2
\,dt'
\\ & = \frac{\eps_0^2}{2}\left(1 + \frac43\gamma\right) + \mathcal{J}_\infty(\varphi^{\eps_0})
\\ & < \frac{\eps_0^2}{2}\left(1 + \frac43\gamma\right) + (C_{\eps_0}) + \eps,
\end{aligned}
\]
where we used the change of variable $t'=t-1$. By taking the limit $\eps_0,\eps\to0$
we obtain that $(C)\leq (C_{\eps_0})$, which turns~\eqref{eq:ineqeps0} into
\[
\underset{T\to +\infty}{\limsup}\, \eps_T \log\proba_{\eps_0}\left(
\int_{0}^{\tilde \tau}f_p(X_s^T)\,ds \geq (x-\delta)\right)
\leq - (x-\delta)^{\frac2p}J_\infty^\star.
\]
Taking the limit $\delta \to 0$ concludes the proof of Lemma~\ref{lem:CT}.

\section{Long time behaviour of variational problem}
\label{sec:cutting}

In this section we follow the simple approximation argument used
in~\cite[Lemma~3.1]{bazhba2022large} to prove that, for any $x_0\in\R$, it holds
\[
\underset{H\to\infty}{\lim}
\inf_{\varphi\in \mathscr{C}_{x_0}([0,H])}\left\{
\mathcal{J}_H(\varphi),\quad \int_0^H f_p(\varphi_t)\,dt=1
\right\} = \inf_{\varphi\in \mathscr{C}_{x_0}([0,+\infty])}\left\{
\mathcal{J}_\infty(\varphi),\quad \int_0^\infty f_p(\varphi_t)\,dt=1
\right\}.
\]
Thanks to Lemma~\ref{lem:FW} in Appendix~\ref{sec:instanton}, this is equivalent to proving that
\[
\underset{H\to\infty}{\lim}
\underbrace{\inf_{\varphi\in \mathscr{C}_{x_0}([0,H])}\left\{
\mathcal{J}_H(\varphi),\quad \int_0^H |\varphi_t|^p\,dt \geq 1
\right\}}_{(A_H)} =
\underbrace{\inf_{\varphi\in \mathscr{C}_{x_0}([0,+\infty])}\left\{
\mathcal{J}_\infty(\varphi),\quad \int_0^\infty |\varphi_t|^p\,dt \geq 1 
\right\}}_{(B)}.
\]
Let us start proving that $\lim_{H\to\infty}(A_H)\geq (B)$
by considering a minimizer~$\varphi^H$ of~$(A_H)$
(or quasi-minimizer equivalently) on~$[0,H]$ that we extend to~$[0,+\infty)$ through
  \[
  \forall\,t\in[0,+\infty),\quad
    \tilde \varphi^H_t = \varphi^H_t \ind_{\{ t\leq H\}} + \varphi^H_H
    \e^{ - \gamma (t-H)}\ind_{\{ t> H\}}.
\]
The function~$\tilde \varphi^H$ belongs to~$\mathscr{C}_{x_0}([0,+\infty])$,
satisfies the constraint since $\overline F_p^\infty(\tilde \varphi^H)\geq\overline F_p^H( \varphi^H)\geq 1$, 
and we have
\[
\mathcal{J}_\infty(\tilde \varphi^H)
=\mathcal{J}_H(\varphi^H) +\frac{(\varphi^H_H)^2}{2} \int_H^{\infty}
\underbrace{\left|- \gamma \e^{ - \gamma (t-H)} + \gamma\e^{ - \gamma (t-H)}\right|^2}_{=0}\,dt.
\]
Hence
\[
\mathcal{J}_H(\varphi^H) = \mathcal{J}_\infty(\tilde \varphi^H)\geq (B).
\]
This entails that $\lim_{H\to\infty}(A_H)\geq (B)$.

Let us now turn to $\lim_{H\to\infty}(A_H)\leq (B)$ by considering~$\varphi$ a (quasi) minimizer
of~$(B)$ and introducing
\[
\forall\,t\in[0,H],\quad
\varphi^H_t = c_H \varphi_t \ind_{\{ t\leq H\}},
\]
where the factor
\[
c_H = \left( \int_0^H |\varphi_t|^p\,dt\right)^{-\frac1p}
\]
ensures that the constraint $\overline F_p^H( \varphi^H)=\overline F_p^\infty( \varphi^H)= 1$
is satisfied. A simple calculation shows that:
\[
(B) =\mathcal{J}_\infty(\varphi) \geq c_H^{-2} \mathcal{J}_\infty(\varphi^H) =
c_H^{-2} \int_0^H |\dot \varphi_t^H + \gamma \varphi_t^H|^2 \,dt\geq c_H^{-2} (A).
\]
Due to the constrain on~$\varphi$ it holds~$c_H\xrightarrow[]{H\to+\infty} c_\infty\leq 1$
so by taking the limit~$H\to\infty$ we obtain that $(B)\geq\lim_{H\to\infty} (A_H)$
and the proof is complete.

\section{Large deviations of the number of cycles}
\label{sec:MT}

In this section we consider~$N_T$ defined in~\eqref{eq:NT} and introduce the shorthand
notation:
\[
\forall\,k\geq 1,\quad S_k = \sum_{i=1}^k \tau_i.
\]
Let us compute an exponential bound on~$\proba(N_T\notin \mathcal M_{T,\bar \eps})$
by following now~\cite[Section~5]{bazhba2022large}.
For any $\bar \eps >0$ we write
\[
\begin{aligned}
  \proba(N_T\notin \mathcal M_{T,\bar \eps}) &=
  \proba\left( \left| \frac{N_T}{T} - \frac{1}{\E[\tau]}\right| \geq \bar \eps
  \right)
  \\ & = \proba\left(  N_T \geq T\left(\frac{1}{\E[\tau]}+ \bar \eps\right)
  \right) +\proba\left(  N_T \leq T\left(\frac{1}{\E[\tau]}- \bar \eps\right)
  \right).
\end{aligned}
\]
By symmetry we can restrict to studying the first probability in the last line
above. By introducing
\[
T' = T\left(\frac{1}{\E[\tau]}+ \bar \eps\right)
\]
we obtain
\[
\begin{aligned}
\proba\left(  N_T \geq T\left(\frac{1}{\E[\tau]}+ \bar \eps\right)
\right) & = \proba\left( \max\,\{k,\, S_k \leq T\}\geq T\left(\frac{1}{\E[\tau]}+ \bar \eps\right)
\right) \\ & =
\proba \left( S_{\left\lfloor T\left(\frac{1}{\E[\tau]}+ \bar \eps\right)\right\rfloor }
\leq T \right)
\\ & = \proba \left( \frac{S_{\left\lfloor T'\right\rfloor }}{T'}
\leq\frac{\E[\tau]}{ 1 + \bar \eps\E[\tau]}  \right).
\end{aligned}
\]
We recall that
\[
\frac{S_{\left\lfloor T'\right\rfloor }}{T'}
= \frac{1}{T'} \sum_{i=1}^{\lfloor T' \rfloor} \tau_i,
\]
and note that for any $\bar \eps >0$ it holds
\[
\frac{\E[\tau]}{ 1 + \bar \eps\E[\tau]} < \E[\tau].
\]
Since~$\tau$ is the sum of two stopping times of an Ornstein--Uhlenbeck process, we easily show
that
\[
\E\left[\e^{ \tau}\right]<+\infty.
\]
Therefore, since the~$\tau_i$ are indenpendent
and identically distributed, Cramer's theorem states that there exists $c_{\bar \eps}>0$
such that
\[
\underset{T'\to\infty}{\limsup}\, \frac{1}{T'}\log \proba \left( \frac{S_{\left\lfloor T'\right\rfloor }}{T'}
\leq\frac{\E[\tau]}{ 1 + \bar \eps\E[\tau]}  \right)
\leq - c_{\bar \eps}.
\]
Using a simple change of variable and the computations above, this provides
\[
\underset{T\to\infty}{\limsup}\, \frac1T\log  \proba(N_T\notin \mathcal M_{T,\bar \eps} )
\leq - c_{\bar \eps}.
\]
We therefore obtain that
\[
\underset{T\to\infty}{\limsup}\, \eps_T^2\log  \proba(N_T\notin \mathcal M_{T,\bar \eps} )
 = - \infty
\]
so the deviations of the number of cycles can be neglected at the scale we are interested in.

\section*{Acknowledgements}
The author is deeply grateful towards Hugo Touchette for the so many discussions on
large deviations that for sure helped to build his understanding of the problem
under consideration.

\bibliographystyle{abbrv}

\end{document}